\pdfoutput=1

\documentclass[pdflatex,sn-mathphys-num]{sn-jnl}

\usepackage{amsmath,amssymb,amsfonts}%
\usepackage{amsthm}%
\usepackage{tikz}%

\DeclareMathOperator{\artanh}{artanh}
\DeclareMathOperator{\arcoth}{arcoth}

\theoremstyle{thmstyleone}%
\newtheorem{theorem}{Theorem}%
\newtheorem{lemma}{Lemma}%
\newtheorem{corollary}{Corollary}%

\newcommand*{\VoRDOI}{}
\newcommand{\preprintnotice}{%
  \begingroup
  \renewcommand{\thefootnote}{}%
  \footnotetext{\footnotesize
    This preprint has not undergone peer review or any post-submission improvements or corrections.
    \ifx\VoRDOI\empty
      A revised version of this article has been accepted for publication in \emph{Operations Research Forum}.%
    \else
      The Version of Record of this article is published in \emph{Operations Research Forum}, and is available online at \href{https://doi.org/\VoRDOI}{https://doi.org/\VoRDOI}.%
    \fi}%
  \endgroup}
\makeatletter
\g@addto@macro\@maketitle{\vskip-36pt\preprintnotice\vskip36pt}
\makeatother

\begin{document}

\title[Narrow Corridor in an Extended Osipov--Lanchester Model]{The Narrow Corridor of Stable Solutions in an Extended Osipov--Lanchester Model with Constant Total Population}

\author*[1]{\fnm{Sergey} \sur{Salishev}}\email{salishev@aifoundry.org}

\affil*[1]{\orgname{aifoundry.org}, \orgaddress{\city{San Francisco}, \state{CA}, \country{USA}}}

\abstract{This paper studies a modification of the classical Osipov--Lanchester model in which the total population of the two sides, $N=R+B$, stays constant. In this model the ratio $y=R/B$ obeys the Riccati equation $\dot y=\alpha y^2-\beta$, which is solvable in elementary functions for all values of the parameters. The main result is a complete classification of the stable regimes. An asymptotically stable invariant set in the quadrant $R,B\ge 0$ exists only in three cases: $\alpha<0,\ \beta<0$ gives a stable interior equilibrium; $\alpha=0,\ \beta<0$ makes the boundary $B=0$ stable; $\alpha<0,\ \beta=0$ makes the boundary $R=0$ stable. If $\alpha>0$ or $\beta>0$, every solution except one unstable equilibrium reaches a boundary of the quadrant in finite time, where the model stops being applicable. If the parameters vary in time, a fixed distance from the boundaries is kept when $\alpha(t)$ and $\beta(t)$ obey buffer dynamics that push them away from zero. Finally, in the case $\alpha<0,\ \beta<0$ the original system grows exponentially, while in the two degenerate cases it grows at most linearly. Growth demanded of a given cost, the sum of the two channels, bounds their imbalance and, with the perturbation bound, confines the parameters to an explicit wedge about $|\alpha|=|\beta|$ truncated near the origin: the narrow corridor of the model.}

\keywords{Osipov--Lanchester model, Riccati equation, asymptotic stability, invariant sets, narrow corridor, state--society interaction}

\pacs[MSC Classification]{34A34, 34D20, 91A22, 91D10}

\maketitle

\section{Introduction}\label{sec1}

The book ``The Narrow Corridor''~\cite{AcemogluRobinson2019} describes confrontation and cooperation between the state and society. It assumes that a stable balance of forces exists and that transitions to extreme regimes are possible. The book gives many historical examples, but it does not state a mathematical model and does not study one. It is therefore natural to compare this informal picture with a simple model that shows the same qualitative behavior.

The classical Osipov--Lanchester model~\cite{Lanchester1916,Helmbold1993} with $\alpha>0$, $\beta>0$ describes mutual attrition of two sides and has no stable solutions. The question is how to extend it to all values of $\alpha$ and $\beta$ so that the behavior for $\alpha>0,\ \beta>0$ does not change, while stable solutions appear for $\alpha\le 0$ or $\beta\le 0$. Such solutions correspond to the cooperative regimes of~\cite{AcemogluRobinson2019}. This follows Arnold's view on soft models in the social sciences~\cite{Arnold2000HardSoft}.

The paper proposes the extended Osipov--Lanchester model with constant total population
\begin{equation}\label{eq:system}
\dot R=-\beta B+\mu R,\qquad
\dot B=-\alpha R+\mu B,\qquad
N=R+B=\text{const},
\end{equation}
where the term $\mu(t)$ is chosen so that $N$ stays constant; Section~\ref{sec2} derives it.

The passage from the absolute variables $(R,B)$ to the shares $r=R/N$ and $b=B/N$ is a standard normalization: it separates the composition of the total from its scale. The scale is studied separately in the original system and gives the ``cooperation premium'' of Theorem~\ref{thm:premium}.

The two channels are bought out of one stock. In the reading of Section~\ref{sec6}, $|\alpha|$ is the inclusiveness of the state and $|\beta|$ the patriotism of society, two forms of cooperation drawn from the same population, since the state is staffed from the society it governs. Section~\ref{sec5} shows what that costs. Growth is the geometric mean of the two channels, while the stock they consume is their sum, so an unequal pair spends as much as an equal one and grows less. What narrows the corridor is that imbalance, and being a ratio it is scale-free: a structural property of the pair, unchanged by any common rescaling.

Several mathematical models of the interaction between the state and society are known: extensions of the Lanchester equations to guerrilla, irregular and multi-party conflicts~\cite{Deitchman1962,Kress2020Irregular,Kress2018Threeway}, dynamic models of protest~\cite{Volpert2024}, and the structural-demographic theory of macro-social cycles~\cite{Turchin2003}. Equations of the form~\eqref{eq:system} also appear in nonzero-sum differential games with growth or reinforcements~\cite{DocknerJorgensenLongSorger2000,YeungPetrosyan2012}, but there $\mu$ describes external factors and does not depend on the state. For $N=1$ the system is of replicator type, with $-\mu$ in the role of the mean payoff, though its payoffs are singular at the boundary and do not keep it invariant; a replicator equation with a fixed payoff matrix is in turn equivalent to a Lotka--Volterra system~\cite{Bomze1995NewIssues}. Compared with the general theory of such models~\cite{HofbauerSigmund1998,Sandholm2010}, the reduction to the Riccati equation gives a complete description of the solutions for all parameter values~\cite{Arnold1992OrdinaryDifferentialEquations}.

A closed constant-sum system with an explicit classification of its stable regimes appears to be new. The model keeps the structure of the Osipov--Lanchester equations, remains solvable in closed form, and covers both confrontation and cooperation. Political and economic terms below are labels for mathematically defined regimes.

\paragraph{Main results.}
\begin{enumerate}
\item For $R,B\ge 0$ and constant $N=R+B$, asymptotically stable invariant sets exist only for $\alpha\le 0$ and $\beta\le 0$ (Theorem~\ref{thm:classification}). There are three such cases: a unique interior equilibrium for $\alpha<0,\ \beta<0$; the boundary $B=0$ for $\alpha=0,\ \beta<0$; the boundary $R=0$ for $\alpha<0,\ \beta=0$. If $\alpha>0$ or $\beta>0$, every solution except one unstable equilibrium reaches a boundary in finite time, as in the classical model.
\item If the parameters vary in time, a fixed distance from the boundaries is kept under buffer dynamics that keep $\alpha(t)$ and $\beta(t)$ away from zero (Corollary~\ref{cor:buffer}). For bounded perturbations of constant parameters, Theorem~\ref{thm:robust} gives an explicit stability condition and an error bound. Robustness requires each channel to have a minimum strength, but it bounds the ratio of the two only weakly: the admissible parameters form a subquadrant, not a corridor.
\item For $\alpha<0,\ \beta<0$ and unbounded resources, the solutions of the original Osipov--Lanchester system grow exponentially, while in the degenerate cases they grow at most linearly (Theorem~\ref{thm:premium}). In the model it is growth, not perturbation, that makes the corridor narrow: the growth the two channels return per unit of cost falls as they become unequal, so growth demanded of a given cost bounds their imbalance and confines the parameters to an explicit wedge about $|\alpha|=|\beta|$, truncated near the origin by the perturbations (Corollary~\ref{cor:wedge}). The wedge thins as the bound on the imbalance tightens, and its truncation falls at the same time.
\end{enumerate}

\section{The Model and Its Reduction}\label{sec2}

The classical Osipov--Lanchester model is
\begin{equation}\label{eq:OL}
\dot R=-\beta B,\qquad
\dot B=-\alpha R .
\end{equation}
It is meaningful only for $\alpha>0$, $\beta>0$, and for such parameters it has no stable solutions~\cite{Arnold2000HardSoft}: every trajectory in the positive quadrant reaches a coordinate axis in finite time, apart from the separatrix $\alpha R^2=\beta B^2$ on which both sides decay to zero; either way it is mutual attrition.

For other signs of the parameters~\eqref{eq:OL} no longer describes attrition: for $\alpha<0,\ \beta<0$ both sides grow without bound, while for $\alpha\beta<0$ one of them leaves the quadrant $R,B\ge 0$, since $\alpha>0,\ \beta<0$ gives $\dot R=|\beta|B>0$ and $\dot B\le-\alpha R(0)<0$ as long as $B>0$, and $\alpha<0,\ \beta>0$ is symmetric. In all these cases it is the composition of the total rather than its size that remains meaningful, so the analysis passes to relative sizes. Put
\[
N=R+B,\qquad r=\frac{R}{N},\qquad b=\frac{B}{N},\qquad r+b=1 .
\]
Then
\begin{equation}\label{eq:Ndot}
\dot N=\dot R+\dot B=-\beta B-\alpha R=-\mu N,
\qquad \mu:=\alpha r+\beta b .
\end{equation}
The quotient rule and~\eqref{eq:Ndot} give the dynamics of the shares
\begin{equation}\label{eq:shares_system}
\dot r=\frac{\dot R}{N}-r\frac{\dot N}{N}=-\beta b+\mu r,
\qquad
\dot b=\frac{\dot B}{N}-b\frac{\dot N}{N}=-\alpha r+\mu b .
\end{equation}

System~\eqref{eq:shares_system} is exactly system~\eqref{eq:system} written for $N=1$. This is the reason for the chosen extension: the shares of the classical model obey the extended model, which is thus the classical model with the scale divided out. In the absolute variables it keeps the interaction terms of~\eqref{eq:OL} and adds the term $\mu(t)$ of~\eqref{eq:Ndot}, fixed by the condition $\dot N=\dot R+\dot B=0$,
\begin{equation}\label{eq:mu}
\mu(t)=\frac{\alpha R+\beta B}{N}.
\end{equation}
Thus $\mu(t)$ is a normalizing feedback: it cancels the change of scale and leaves only the dynamics of the relative variables. Dropping it, that is putting $\mu\equiv 0$ in~\eqref{eq:system}, returns the classical model~\eqref{eq:OL}, so the parameters $\alpha$ and $\beta$ keep their meaning in both.

The rest of the paper works with two relative variables: the share
\[
x=\frac{R}{N}\in[0,1]
\]
and, while $B>0$, the ratio
\[
y=\frac{R}{B}=\frac{x}{1-x}\in(0,\infty),
\qquad\text{so that}\qquad
x=\frac{y}{1+y} .
\]
\begin{lemma}[Dynamics of the share and the ratio]\label{lem:x-y-dynamics}
Let $\alpha,\beta\in\mathbb{R}$, let $\mu(t)$ be given by~\eqref{eq:mu}, and let $(R,B)$ be a solution of~\eqref{eq:system} with $N=R+B>0$ constant. Then
\begin{equation}\label{eq:x}
\dot x=(\alpha-\beta)x^2+2\beta x-\beta=\alpha x^2-\beta(1-x)^2 ,
\end{equation}
and, as long as $B>0$,
\begin{equation}\label{eq:riccati}
\dot y=\alpha y^2-\beta .
\end{equation}
Let moreover $\alpha\neq 0$ and put
\begin{equation}\label{eq:equilibrium}
y_*:=\sqrt{\left|\frac{\beta}{\alpha}\right|},
\qquad
x_*:=\frac{y_*}{1+y_*}=\frac{\sqrt{|\beta|}}{\sqrt{|\alpha|}+\sqrt{|\beta|}} .
\end{equation}
Then $y\equiv y_*$ is an equilibrium of~\eqref{eq:riccati}, and $x\equiv x_*$ an equilibrium of~\eqref{eq:x}, if and only if $\alpha\beta\ge 0$. It is interior, that is $y_*>0$ and $x_*\in(0,1)$, if and only if $\alpha\beta>0$.
\end{lemma}

\begin{proof}
Since $N$ is constant, $\dot x=\dot R/N$. Using~\eqref{eq:system} and $B/N=1-x$,
\begin{equation}\label{eq:dx-step}
\dot x=\frac{-\beta B+\mu R}{N}=-\beta(1-x)+\mu x .
\end{equation}
By~\eqref{eq:mu},
\[
\mu=\alpha\frac{R}{N}+\beta\frac{B}{N}=\alpha x+\beta(1-x).
\]
Substituting this into~\eqref{eq:dx-step} gives
\[
\dot x=-\beta(1-x)+x\bigl[\beta+(\alpha-\beta)x\bigr]
=(\alpha-\beta)x^2+2\beta x-\beta ,
\]
and grouping the terms as $\alpha x^2-\beta(x^2-2x+1)$ gives the second form in~\eqref{eq:x}.

For the ratio, with $B>0$, differentiation and~\eqref{eq:system} give
\[
\dot y=\frac{\dot R\,B-R\,\dot B}{B^2}
=\frac{(-\beta B+\mu R)B-R(-\alpha R+\mu B)}{B^2}
=\frac{\alpha R^2-\beta B^2}{B^2}
=\alpha y^2-\beta .
\]

Finally, let $\alpha\neq 0$. Since $y_*^2=|\beta|/|\alpha|$ and $1-x_*=\sqrt{|\alpha|}/\bigl(\sqrt{|\alpha|}+\sqrt{|\beta|}\bigr)$,
\[
\alpha y_*^2-\beta=\frac{\alpha}{|\alpha|}\,|\beta|-\beta ,
\qquad
\alpha x_*^2-\beta(1-x_*)^2=\frac{\alpha|\beta|-\beta|\alpha|}{\bigl(\sqrt{|\alpha|}+\sqrt{|\beta|}\bigr)^2} .
\]
Both right-hand sides equal $|\beta|-\beta$ times a positive factor when $\alpha>0$, and $-\bigl(|\beta|+\beta\bigr)$ times a positive factor when $\alpha<0$, so both vanish exactly when $\beta\ge 0$ in the first case and $\beta\le 0$ in the second, that is, exactly when $\alpha\beta\ge 0$. Finally $y_*>0$, equivalently $x_*\in(0,1)$, if and only if $\beta\neq 0$; together with $\alpha\beta\ge 0$ this is $\alpha\beta>0$.
\end{proof}

In the last computation the terms with $\mu$ cancel: the ratio does not feel the normalizing feedback, which is why~\eqref{eq:riccati} carries no trace of the constant-sum constraint. The same computation with $\mu\equiv 0$ applies to the classical model~\eqref{eq:OL} and gives the same equation, so the extension changes only the scale $N$ and leaves the relative behavior unchanged for every $\alpha$ and $\beta$, in particular in the classical regime $\alpha>0,\ \beta>0$. This settles the requirement stated in Section~\ref{sec1}. The Riccati equation~\eqref{eq:riccati} is the object studied in the rest of the paper.

Equation~\eqref{eq:x} is symmetric in the two sides: if $x$ solves it for the parameters $(\alpha,\beta)$, then $1-x$ solves it for $(\beta,\alpha)$, because $\alpha x^2-\beta(1-x)^2$ changes sign under the substitution. The exchange $x\mapsto 1-x$, $\alpha\leftrightarrow\beta$ thus maps the model to itself with the roles of the two sides interchanged.

\section{Solutions and Classification of Stable Regimes}\label{sec3}

\subsection{Invariance of the domain}\label{subsec3.1}

The first question about the domain is when the share stays in $[0,1]$, that is, when both sides remain nonnegative. The reduction of Lemma~\ref{lem:x-y-dynamics} does not use constancy of the parameters, so the question is asked at once for time-dependent ones. Write
\begin{equation}\label{eq:star}
\dot x=f(t,x):=\alpha(t)x^2-\beta(t)(1-x)^2 ,
\end{equation}
with $\alpha,\beta$ continuous functions of $t$, constants included, so that~\eqref{eq:star} has solutions. This is~\eqref{eq:x} with $\alpha,\beta$ replaced by $\alpha(t),\beta(t)$: the derivation of~\eqref{eq:x} is an identity between the values of $\dot R$, $\mu$ and $x$ at a single instant, so it holds at each $t$ for time-dependent parameters as well.

A solution leaves the interval by crossing one of its ends. The next lemma describes the crossing.

\begin{lemma}[First crossing of a level]\label{lem:exit}
Let $J$ be an interval with least element $0$, let $u\in C^1(J)$ and $c\in\mathbb{R}$ be such that $u(0)\ge c$ while $u(t_1)<c$ for some $t_1\in J$. Then there is $t^*\in[0,t_1)$ with
\[
u(t^*)=c,\qquad u(t)<c\quad\text{for all }t\in(t^*,t_1],\qquad \dot u(t^*)\le 0 .
\]
In particular, if $\dot u(t)>0$ at every $t\in J$ with $u(t)=c$, then $u(t)\ge c$ on all of $J$.
\end{lemma}

\begin{proof}
The set $\{t\in[0,t_1]:u(t)\ge c\}$ is closed and contains $0$, so it has a largest element $t^*$, and $t^*<t_1$ because $u(t_1)<c$. By maximality $u<c$ on $(t^*,t_1]$, so $u(t^*)=\lim_{t\downarrow t^*}u(t)\le c$; together with $u(t^*)\ge c$ this gives $u(t^*)=c$. Finally
\[
\dot u(t^*)=\lim_{t\downarrow t^*}\frac{u(t)-c}{t-t^*}\le 0 ,
\]
because the numerator is negative and the denominator positive. For the last claim, $\dot u(t^*)\le 0$ at a point where $u(t^*)=c$ contradicts the hypothesis.
\end{proof}

The last claim of Lemma~\ref{lem:exit} needs the velocity at the level to be strictly inward, while at the ends of $[0,1]$ the model gives $f(t,0)=-\beta(t)\ge 0$ and $f(t,1)=\alpha(t)\le 0$. An estimate on the excursion outside the interval closes the gap.

\begin{lemma}[Invariance]\label{lem:inv}
Let $\alpha,\beta$ be continuous with $\alpha(t)\le 0$ and $\beta(t)\le 0$ for all $t\ge 0$. Then $[0,1]$ is forward invariant for~\eqref{eq:star}.
\end{lemma}

\begin{proof}
Let $x(0)\in[0,1]$ and let $J$ be the interval on which the solution is defined.

Suppose $x(t_1)<0$ for some $t_1\in J$, and take $t^*$ as in Lemma~\ref{lem:exit} with $u=x$ and $c=0$. Since $x(t^*)=0$ and $x$ is continuous, moving $t_1$ towards $t^*$ one may assume in addition that $x\ge-1$ on $(t^*,t_1]$. Put $v:=-x$ and $K:=\max_{[t^*,t_1]}|\alpha|$, so that $0<v\le 1$ there. Since $\beta\le 0$ and $\alpha\le 0$,
\[
\dot v=-\alpha x^2+\beta(1-x)^2\le|\alpha|\,v^2\le Kv .
\]
Hence $\frac{d}{dt}\bigl(e^{-Kt}v\bigr)\le 0$, so $e^{-Kt_1}v(t_1)\le e^{-Kt^*}v(t^*)=0$ and $v(t_1)\le 0$, contradicting $v(t_1)>0$.

The other end is symmetric. Suppose $x(t_1)>1$ and take $t^*$ from Lemma~\ref{lem:exit} with $u=1-x$ and $c=0$, so that $x(t^*)=1$ and $x>1$ on $(t^*,t_1]$; as above assume also $x\le 2$ there. Put $w:=x-1$ and, again, $K:=\max_{[t^*,t_1]}|\beta|$, so that $0<w\le 1$. Since $\alpha\le 0$ and $\beta\le 0$,
\[
\dot w=\alpha x^2-\beta(1-x)^2\le|\beta|\,w^2\le Kw ,
\]
and the same integrating factor gives $w(t_1)\le 0$, again a contradiction. Hence $x(t)\in[0,1]$ for as long as the solution exists.
\end{proof}

\subsection{Explicit solutions}\label{subsec3.2}

From here to the end of Section~\ref{sec3} the parameters are constant. Equation~\eqref{eq:riccati} is then autonomous and separable, so it integrates in elementary functions for every sign combination of $\alpha$ and $\beta$.

\begin{theorem}[Solutions for $y=R/B$]\label{thm:solutions}
Consider $\dot y=\alpha y^2-\beta$ with $y(0)=y_0>0$, let $y_*$ be as in~\eqref{eq:equilibrium}, and put
\[
\kappa:=\sqrt{|\alpha\beta|},
\qquad
\arcoth u:=\artanh\frac{1}{u}\quad(|u|>1).
\]
\begin{enumerate}
\item If $\alpha>0$ and $\beta>0$, then $y\equiv y_*$ is an equilibrium, and
\begin{equation}\label{eq:y_pp}
y(t)=
\begin{cases}
y_*\coth\bigl(\arcoth(y_0/y_*)-\kappa t\bigr), & y_0>y_*,\\[1mm]
y_*\tanh\bigl(\artanh(y_0/y_*)-\kappa t\bigr), & y_0<y_* .
\end{cases}
\end{equation}
For $y_0>y_*$ the solution increases and tends to $+\infty$ as $t\to T=\arcoth(y_0/y_*)/\kappa$; for $y_0<y_*$ it decreases and vanishes at $T=\artanh(y_0/y_*)/\kappa$.

\item If $\alpha<0$ and $\beta<0$, then $y\equiv y_*$ is an equilibrium, and
\begin{equation}\label{eq:y_mm}
y(t)=
\begin{cases}
y_*\tanh\bigl(\artanh(y_0/y_*)+\kappa t\bigr), & y_0<y_*,\\[1mm]
y_*\coth\bigl(\arcoth(y_0/y_*)+\kappa t\bigr), & y_0>y_* .
\end{cases}
\end{equation}
Every solution is defined for all $t\ge 0$ and converges monotonically to $y_*$.

\item If $\alpha>0$ and $\beta<0$, then
\begin{equation}\label{eq:y_pn}
y(t)=y_*\tan\bigl(\arctan(y_0/y_*)+\kappa t\bigr).
\end{equation}
The solution increases and tends to $+\infty$ as $t\to T=\bigl(\pi/2-\arctan(y_0/y_*)\bigr)/\kappa$.

\item If $\alpha<0$ and $\beta>0$, then
\begin{equation}\label{eq:y_np}
y(t)=y_*\tan\bigl(\arctan(y_0/y_*)-\kappa t\bigr).
\end{equation}
The solution decreases and vanishes at $T=\arctan(y_0/y_*)/\kappa$.

\item If $\alpha=0$, then $y(t)=y_0-\beta t$; it vanishes at $T=y_0/\beta$ when $\beta>0$ and grows linearly when $\beta<0$. If $\beta=0$, then
\[
y(t)=\frac{y_0}{1-\alpha y_0 t};
\]
it tends to $+\infty$ as $t\to T=1/(\alpha y_0)$ when $\alpha>0$, and decreases to $0$ as $t\to\infty$ when $\alpha<0$.
\end{enumerate}
\end{theorem}

\begin{proof}
The right-hand side of~\eqref{eq:riccati} is a polynomial in $y$, hence locally Lipschitz, so the solution with $y(0)=y_0$ is unique. It is enough to integrate the equation on each side of an equilibrium and to check the initial value. Note that $\alpha y_*=\kappa$ for $\alpha>0$ and $|\alpha|y_*=\kappa$ for $\alpha<0$.

\emph{Case $\alpha>0,\ \beta>0$.} Here $\dot y=\alpha(y^2-y_*^2)$, so $y\equiv y_*$ is an equilibrium and no solution crosses it. Separating the variables gives
\[
\int\frac{dy}{y^2-y_*^2}=\alpha t+C .
\]
For $0<y<y_*$ the antiderivative on the left is $-\artanh(y/y_*)/y_*$, and for $y>y_*$ it is $-\arcoth(y/y_*)/y_*$. Multiplying by $-y_*$ and using $\alpha y_*=\kappa$,
\[
\artanh\frac{y(t)}{y_*}=\artanh\frac{y_0}{y_*}-\kappa t
\quad\text{or}\quad
\arcoth\frac{y(t)}{y_*}=\arcoth\frac{y_0}{y_*}-\kappa t ,
\]
which is~\eqref{eq:y_pp}. In the first case the argument of $\tanh$ reaches $0$ at $t=\artanh(y_0/y_*)/\kappa$, where $y=0$. In the second case the argument of $\coth$ reaches $0$ at $t=\arcoth(y_0/y_*)/\kappa$, where $y$ blows up.

\emph{Case $\alpha<0,\ \beta<0$.} Here $\dot y=|\alpha|(y_*^2-y^2)$ with $|\alpha|y_*=\kappa$. The same antiderivatives with the opposite sign give
\[
\artanh\frac{y(t)}{y_*}=\artanh\frac{y_0}{y_*}+\kappa t
\quad\text{or}\quad
\arcoth\frac{y(t)}{y_*}=\arcoth\frac{y_0}{y_*}+\kappa t ,
\]
which is~\eqref{eq:y_mm}. Both arguments stay away from $0$ and grow, so the solution exists for all $t\ge 0$ and tends to $y_*$.

\emph{Case $\alpha>0,\ \beta<0$.} Here $\dot y=\alpha y^2+|\beta|=|\beta|\bigl(1+(y/y_*)^2\bigr)>0$ and
\[
\int\frac{dy}{\alpha y^2+|\beta|}=\frac{y_*}{|\beta|}\arctan\frac{y}{y_*}=t+C .
\]
Since $y_*/|\beta|=1/\kappa$, this gives $\arctan\bigl(y(t)/y_*\bigr)=\arctan(y_0/y_*)+\kappa t$, that is,~\eqref{eq:y_pn}. The solution exists until the argument reaches $\pi/2$.

\emph{Case $\alpha<0,\ \beta>0$.} Here $\dot y=-(|\alpha| y^2+\beta)<0$ and
\[
\int\frac{dy}{|\alpha| y^2+\beta}=\frac{y_*}{\beta}\arctan\frac{y}{y_*}=-t+C .
\]
Since $y_*/\beta=1/\kappa$, this gives $\arctan\bigl(y(t)/y_*\bigr)=\arctan(y_0/y_*)-\kappa t$, that is,~\eqref{eq:y_np}. The solution reaches $y=0$ when the argument reaches $0$.

\emph{Degenerate cases.} If $\alpha=0$, then $\dot y=-\beta$ and $y(t)=y_0-\beta t$. If $\beta=0$, then $\dot y=\alpha y^2$ and $y(t)=y_0/(1-\alpha y_0t)$. The stated limits follow directly.
\end{proof}

\subsection{Classification of stable invariant sets}\label{subsec3.3}

Write $g(y):=\alpha y^2-\beta$ for the right-hand side of~\eqref{eq:riccati}, so that the interior equilibria are its positive roots; by Lemma~\ref{lem:x-y-dynamics} there is one exactly when $\alpha\beta>0$, namely $y_*$, with share $x_*$.

\begin{theorem}[Classification of stable regimes]\label{thm:classification}
Consider system~\eqref{eq:system} with $R,B\ge 0$ and $N>0$.
\begin{enumerate}
\item The domain $R,B\ge 0$ is forward invariant if and only if $\alpha\le 0$ and $\beta\le 0$.
\item The interior equilibrium~\eqref{eq:equilibrium}, which exists exactly when $\alpha\beta>0$, is asymptotically stable for $\alpha<0,\ \beta<0$ and unstable for $\alpha>0,\ \beta>0$.
\item If $\alpha=0$ and $\beta<0$, the only equilibrium in the domain is the boundary $B=0$; if $\alpha<0$ and $\beta=0$, it is the boundary $R=0$. In both cases the boundary attracts all solutions in the domain, and the convergence is of order $1/t$.
\item If $\alpha>0$ or $\beta>0$, the domain $R,B\ge 0$ is not forward invariant, and every solution except the unstable equilibrium of item~2 reaches the boundary $R=0$ or $B=0$ in finite time.
\item If $\alpha=\beta=0$, the sides do not interact: every point of the domain is an equilibrium, and no proper subset of it is asymptotically stable.
\end{enumerate}
\end{theorem}

\begin{proof}
By Lemma~\ref{lem:x-y-dynamics} the state of the system is described by the share $x$, which obeys
\[
\dot x=\alpha x^2-\beta(1-x)^2 ,
\qquad
\dot x\big|_{x=0}=-\beta,\qquad \dot x\big|_{x=1}=\alpha .
\]

1. The interval $[0,1]$ is forward invariant if and only if the vector field points inward at both ends,
\[
\dot x|_{x=0}\ge 0\iff\beta\le 0,
\qquad
\dot x|_{x=1}\le 0\iff\alpha\le 0 .
\]
If $\beta>0$, the solution starting at $x=0$ leaves the interval, and likewise for $\alpha>0$ at $x=1$; so both conditions are necessary. They are sufficient by Lemma~\ref{lem:inv}, of which constant parameters are a special case.

2. Existence and uniqueness are Lemma~\ref{lem:x-y-dynamics}. The linearization of~\eqref{eq:riccati} at $y_*$ has coefficient $g'(y_*)=2\alpha y_*$, so the equilibrium is asymptotically stable for $\alpha<0$ and unstable for $\alpha>0$. The sign of $\beta$ is the same as that of $\alpha$ because $\alpha\beta>0$.

3. For $\alpha=0$ and $\beta<0$ equation~\eqref{eq:x} becomes $\dot x=|\beta|(1-x)^2\ge 0$, so $x=1$, that is $B=0$, is the only equilibrium in $[0,1]$ and every solution increases to it. Putting $w=1-x$ gives $\dot w=-|\beta|w^2$, hence $w(t)=w(0)/(1+|\beta|w(0)t)$: the convergence is of order $1/t$. For $\beta=0$ and $\alpha<0$ equation~\eqref{eq:x} becomes $\dot x=-|\alpha|x^2\le 0$, and the same computation shows that every solution decreases to $x=0$, that is, to $R=0$, again at rate $1/t$.

4. Let $\alpha>0$ or $\beta>0$. By item~1 the domain is not forward invariant. A solution starting on the boundary is there already, so let $R(0),B(0)>0$ and put $y_0=R(0)/B(0)\in(0,\infty)$. If $\alpha\beta\le 0$, or if $\alpha\beta>0$ and $y_0\neq y_*$, then by Theorem~\ref{thm:solutions} the ratio $y$ either vanishes or blows up in finite time; in the first case the solution reaches $R=0$, in the second it reaches $B=0$. The only solution that stays in the interior is the equilibrium $y\equiv y_*$, which exists for $\alpha>0,\ \beta>0$ and is unstable by item~2.

5. If $\alpha=\beta=0$, then $\dot x\equiv 0$ by~\eqref{eq:x}, so every point is an equilibrium. A proper subset cannot attract its neighbours, which are equilibria as well.
\end{proof}

The hitting times in item~4 are the times $T$ listed in Theorem~\ref{thm:solutions}. The five items cover all values of $\alpha$ and $\beta$, so items~2 and~3 give the only three cases with an asymptotically stable invariant set.

\section{Robustness of the Corridor}\label{sec4}

By Theorem~\ref{thm:classification} the stable regimes fill the quadrant $\alpha\le 0,\ \beta\le 0$ apart from its corner $\alpha=\beta=0$, whereas the corridor of~\cite{AcemogluRobinson2019} is narrow. The parameters have been constant so far, while in practice $\alpha$ and $\beta$ change in time, and a corridor must be stable under such changes~\cite{khalil2002nonlinear}. This section investigates how far that demand narrows the quadrant. Equation~\eqref{eq:star} already allows the parameters to vary; the question is what they must do for the share to stay away from the boundaries, and which pairs $(\alpha,\beta)$ meet the requirement.

\subsection{Parameters varying in time}\label{subsec4.1}

\begin{theorem}[Conditions on the dynamics of the parameters]\label{thm:dyn_rec}
Let $\alpha,\beta\in C^1([0,\infty))$ with $\alpha(0),\beta(0)\in[-1,0]$, and let
\[
\alpha(t)=0\Rightarrow\dot\alpha(t)<0,\qquad
\alpha(t)=-1\Rightarrow\dot\alpha(t)>0,
\]
and the same for $\beta$, for all $t\ge 0$. Then $\alpha(t),\beta(t)\in[-1,0]$ for all $t\ge 0$. Consequently $[0,1]$ is forward invariant for~\eqref{eq:star}, and if $x(0)\in(0,1)$, then $x(t)\in(0,1)$ for all $t\ge 0$, that is, $R(t)>0$ and $B(t)>0$.
\end{theorem}

\begin{proof}
Consider $\alpha$; the argument for $\beta$ is the same. Both ends follow from the last claim of Lemma~\ref{lem:exit}. Applied to $u=\alpha$ with $c=-1$, it requires $\dot\alpha(t)>0$ wherever $\alpha(t)=-1$, which is the second hypothesis, and gives $\alpha(t)\ge-1$. Applied to $u=-\alpha$ with $c=0$, it requires $-\dot\alpha(t)>0$ wherever $\alpha(t)=0$, which is the first hypothesis, and gives $\alpha(t)\le 0$. Hence $\alpha(t)\in[-1,0]$ for all $t\ge 0$, and likewise for $\beta$. Lemma~\ref{lem:inv} then gives the invariance of $[0,1]$.

Let $x(0)\in(0,1)$. Since $\beta(t)\le 0$ and $\alpha(t)\ge -1$,
\[
\dot x=\alpha x^2-\beta(1-x)^2\ge\alpha x^2\ge -x^2 ,
\]
so by the comparison theorem for scalar equations~\cite{khalil2002nonlinear}
\[
x(t)\ge\frac{x(0)}{1+x(0)t}>0 .
\]
Similarly, for $w=1-x$ the bounds $\alpha(t)\le 0$ and $\beta(t)\ge -1$ give $\dot w=-\alpha x^2+\beta w^2\ge -w^2$ and $w(t)\ge w(0)/(1+w(0)t)>0$. Hence $x(t)\in(0,1)$ for all $t\ge 0$.
\end{proof}

The normalization of the parameters to $[-1,0]$ costs nothing: a rescaling of time turns any common lower bound into $-1$. Theorem~\ref{thm:dyn_rec} keeps the trajectory inside the interval, but it allows $x(t)$ to approach the boundaries as $t\to\infty$. To keep a fixed distance from them, the parameters must be pushed away from zero. This is expressed by three constants: the buffer width $\theta$, which is how far from either end of $[-1,0]$ the parameters are held; the return rate $\eta$, which is how fast they are pushed back once inside a buffer; and the margin $\varepsilon$, which is the distance from the boundaries $x=0$ and $x=1$ that the share is then guaranteed to keep. The first two are imposed on the parameters, the third is what they buy.

\begin{corollary}[Buffer dynamics]\label{cor:buffer}
Let $\theta\in(0,1/2)$, so that the ranges $[-\theta,0]$ and $[-1,-1+\theta]$ are disjoint, and let $\eta>0$ be such that for all $t\ge 0$
\[
\begin{aligned}
&\alpha\in[-\theta,0]\Rightarrow\dot\alpha\le-\eta, &&\alpha\in[-1,-1+\theta]\Rightarrow\dot\alpha\ge\eta,\\
&\beta\in[-\theta,0]\Rightarrow\dot\beta\le-\eta, &&\beta\in[-1,-1+\theta]\Rightarrow\dot\beta\ge\eta .
\end{aligned}
\]
If $\alpha(0),\beta(0)\in[-1+\theta,-\theta]$, then $\alpha(t),\beta(t)\in[-1+\theta,-\theta]$ for all $t\ge 0$. Moreover, let the margin $\varepsilon$ satisfy
\[
0<\varepsilon<\frac{\sqrt\theta}{1+\sqrt\theta}<\frac12 .
\]
Then $[\varepsilon,1-\varepsilon]$ is forward invariant for~\eqref{eq:star}, that is, the share keeps the distance $\varepsilon$ from both boundaries.
\end{corollary}

\begin{proof}
The first claim repeats the argument of Theorem~\ref{thm:dyn_rec}. At the lower end of $[-1+\theta,-\theta]$ the hypothesis gives $\dot\alpha\ge\eta>0$, and at the upper end $\dot\alpha\le-\eta<0$, so the last claim of Lemma~\ref{lem:exit}, applied to $u=\alpha$ with $c=-1+\theta$ and to $u=-\alpha$ with $c=\theta$, shows that $\alpha$ cannot leave the interval; the same holds for $\beta$.

Hence $\theta\le|\alpha(t)|\le 1$ and $\theta\le|\beta(t)|\le 1$ for all $t\ge 0$. The choice of $\varepsilon$ is equivalent to $\varepsilon/(1-\varepsilon)<\sqrt\theta$, that is, to $\varepsilon^2<\theta(1-\varepsilon)^2$. At the ends of $[\varepsilon,1-\varepsilon]$ this gives
\[
f(t,\varepsilon)=-|\alpha|\varepsilon^2+|\beta|(1-\varepsilon)^2\ge-\varepsilon^2+\theta(1-\varepsilon)^2>0 ,
\]
\[
f(t,1-\varepsilon)=-|\alpha|(1-\varepsilon)^2+|\beta|\varepsilon^2\le-\theta(1-\varepsilon)^2+\varepsilon^2<0 .
\]
The vector field points strictly inward at both ends, so the last claim of Lemma~\ref{lem:exit}, applied to $u=x$ with $c=\varepsilon$ and to $u=-x$ with $c=-(1-\varepsilon)$, gives the invariance of $[\varepsilon,1-\varepsilon]$.
\end{proof}

The bound on the margin is sharp. Of the admissible pairs, $|\alpha|=1$, $|\beta|=\theta$ drives the share hardest towards $x=0$, and by~\eqref{eq:equilibrium} its equilibrium lies at $\sqrt\theta/(1+\sqrt\theta)$, with $f$ positive to the left of that point and negative to the right; placing $\varepsilon$ to the left of it is thus exactly the condition $f(t,\varepsilon)>0$ for every admissible pair, and the end $x=1-\varepsilon$ gives the same bound by the symmetry of~\eqref{eq:x}. The corollary therefore shows what the buffer dynamics cost: the guaranteed margin $\varepsilon$ shrinks with $\theta$, the strength the buffer guarantees to both channels. The rest of this section makes the trade-off quantitative for bounded perturbations of constant parameters.

\subsection{Bounded perturbations}\label{subsec4.2}

From here to the end of Section~\ref{sec4} the parameters of~\eqref{eq:system} are a constant part plus a bounded perturbation,
\begin{equation}\label{eq:perturbed-params}
\alpha(t)=\alpha_*+\delta_\alpha(t),\qquad
\beta(t)=\beta_*+\delta_\beta(t),
\qquad
|\delta_\alpha(t)|\le\bar\delta_\alpha,\quad|\delta_\beta(t)|\le\bar\delta_\beta,
\end{equation}
with constant nominal values $\alpha_*,\beta_*<0$ and perturbation bounds $\bar\delta_\alpha,\bar\delta_\beta>0$. The analysis is carried out in the share $x$.

For the unperturbed equation $\dot x=f_*(x)$ with
\begin{equation}\label{eq:f-star}
f_*(x):=\alpha_*x^2-\beta_*(1-x)^2
\end{equation}
the equilibrium is the $x_*$ of~\eqref{eq:equilibrium}. The margin $\varepsilon$ has the same meaning as in Corollary~\ref{cor:buffer}, but is now prescribed rather than deduced: fix $\varepsilon\in(0,1/2)$ with $x_*\in(\varepsilon,1-\varepsilon)$, and ask what the parameters must satisfy for the share to keep that margin, that is, to stay in $[\varepsilon,1-\varepsilon]$. This $\varepsilon$ is fixed from now on, and the constants that follow depend on it. Put
\begin{equation}\label{eq:corridor-constants}
m:=\min\{x_*-\varepsilon,\ 1-\varepsilon-x_*\}>0,
\qquad
\gamma:=2\min_{x\in[\varepsilon,1-\varepsilon]}\bigl(|\alpha_*|x+|\beta_*|(1-x)\bigr).
\end{equation}
Here $m$ is the room between the equilibrium and the ends of the interval, and $\gamma$ is the slowest rate at which the share is drawn back to the equilibrium: $f_*'(x)=2\bigl(\alpha_*x+\beta_*(1-x)\bigr)$ is twice the normalizing rate $\mu$ of~\eqref{eq:Ndot} at the nominal parameters, so $f_*'\le-\gamma<0$ throughout. Both constants are invariant under the exchange of the two sides.

The proof of the next theorem uses the following standard estimate.

\begin{lemma}[Differential form of Gronwall's inequality, cf.~\cite{khalil2002nonlinear}]\label{lem:gronwall}
Let $v\colon[0,T]\to\mathbb{R}$ be absolutely continuous, let $p,h\in L^1([0,T])$, and let $\dot v(t)\le p(t)v(t)+h(t)$ for almost all $t\in[0,T]$. Then for all $t\in[0,T]$
\[
v(t)\le v(0)\exp\left(\int_0^tp(\sigma)\,d\sigma\right)
+\int_0^t\exp\left(\int_\sigma^tp(\tau)\,d\tau\right)h(\sigma)\,d\sigma .
\]
\end{lemma}

\begin{theorem}[Robustness under bounded perturbations]\label{thm:robust}
Let the parameters be as in~\eqref{eq:perturbed-params} and let $m,\gamma$ be as in~\eqref{eq:corridor-constants}. If the corridor inequality
\begin{equation}\label{eq:corridor}
\gamma\,m>(\bar\delta_\alpha+\bar\delta_\beta)(1-\varepsilon)^2
\end{equation}
holds, then every solution with $|x(0)-x_*|\le m$ stays in $[\varepsilon,1-\varepsilon]$ for all $t\ge 0$, and its deviation $\psi(t):=x(t)-x_*$ satisfies
\begin{equation}\label{eq:iss-bound}
|\psi(t)|\le e^{-\gamma t}\,|\psi(0)|
+\frac{(\bar\delta_\alpha+\bar\delta_\beta)(1-\varepsilon)^2}{\gamma}\bigl(1-e^{-\gamma t}\bigr),
\qquad t\ge 0 .
\end{equation}
In particular $\limsup_{t\to\infty}|\psi(t)|\le(\bar\delta_\alpha+\bar\delta_\beta)(1-\varepsilon)^2/\gamma<m$.
\end{theorem}

\begin{proof}
By~\eqref{eq:star} the share obeys $\dot x=\alpha(t)x^2-\beta(t)(1-x)^2$. Substituting the perturbed parameters gives
\[
\dot x=f_*(x)+u(t),\qquad
u(t):=\delta_\alpha(t)\,x(t)^2-\delta_\beta(t)\,\bigl(1-x(t)\bigr)^2 .
\]
Since $f_*(x_*)=0$, the mean value theorem gives, for some $\xi$ between $x_*$ and $x(t)$,
\begin{equation}\label{eq:psi-exact}
\dot\psi=f_*'(\xi)\,\psi+u(t),
\qquad
f_*'(\xi)=-2\bigl(|\alpha_*|\xi+|\beta_*|(1-\xi)\bigr).
\end{equation}
As long as $x(t)\in[\varepsilon,1-\varepsilon]$, the point $\xi$ lies in $[\varepsilon,1-\varepsilon]$ as well, so by~\eqref{eq:corridor-constants} the coefficient and the perturbation are bounded by
\[
f_*'(\xi)\le-\gamma<0,
\qquad
|u(t)|\le\bar\delta_\alpha x^2+\bar\delta_\beta(1-x)^2
\le(\bar\delta_\alpha+\bar\delta_\beta)(1-\varepsilon)^2=:\bar u .
\]
The function $|\psi|$ is absolutely continuous. Where $\psi\neq0$ its derivative equals $\operatorname{sign}(\psi)\dot\psi$, and where $\psi=0$ and $|\psi|$ is differentiable its derivative is $0$; in both cases~\eqref{eq:psi-exact} gives
\[
\frac{d}{dt}|\psi(t)|\le-\gamma|\psi(t)|+\bar u
\quad\text{for almost all }t .
\]
Lemma~\ref{lem:gronwall} with $v=|\psi|$, $p\equiv-\gamma$ and $h\equiv\bar u$ yields~\eqref{eq:iss-bound}.

It remains to show that the solution never leaves the interval. Let $t^*$ be the supremum of the times $\tau\ge 0$ such that $x(\sigma)\in[\varepsilon,1-\varepsilon]$ for all $\sigma\in[0,\tau]$; by continuity and $|x(0)-x_*|\le m$ this set is nonempty.

First, $t^*>0$. Indeed, if $x(0)$ is interior to $[\varepsilon,1-\varepsilon]$ this follows from continuity. Otherwise $x(0)$ is an end, and there the drift is strictly inward: at $x=\varepsilon$ the mean value theorem and $f_*(x_*)=0$ give $f_*(\varepsilon)=|f_*'(\xi)|(x_*-\varepsilon)\ge\gamma m>\bar u\ge|u(0)|$ by~\eqref{eq:corridor}, hence $\dot x(0)>0$; at $x=1-\varepsilon$ the same estimate gives $f_*(1-\varepsilon)\le-\gamma m<-|u(0)|$ and $\dot x(0)<0$. Either way the solution enters the interval.

Now suppose $t^*<\infty$. The estimate~\eqref{eq:iss-bound} holds on $[0,t^*]$, and by~\eqref{eq:corridor} $\bar u/\gamma<m$, so for $t\in(0,t^*]$
\[
|\psi(t)|\le e^{-\gamma t}\,m
+\bigl(1-e^{-\gamma t}\bigr)\frac{\bar u}{\gamma}<m .
\]
Hence $x(t^*)$ lies strictly between $x_*-m$ and $x_*+m$, and therefore in the interior of $[\varepsilon,1-\varepsilon]$. By continuity $x$ remains there for slightly longer, which contradicts the definition of $t^*$. So $t^*=\infty$ and $|\psi(t)|\le m$ for all $t\ge 0$.
\end{proof}

\subsection{The explicit corridor condition}\label{subsec4.3}

The constants~\eqref{eq:corridor-constants} still contain the equilibrium $x_*$, which itself depends on $\alpha_*$ and $\beta_*$. Writing them out removes it. Since the corridor inequality~\eqref{eq:corridor} is unchanged by the exchange $\alpha_*\leftrightarrow\beta_*$, one of the two orderings may be assumed; let the first channel be the weaker one,
\begin{equation}\label{eq:normalization}
|\alpha_*|\le|\beta_*| ,
\qquad\text{equivalently}\qquad
y_*\ge 1 ,
\end{equation}
the other case following by the exchange. The two nominal parameters are then described by the size $|\alpha_*|$ of the weaker channel and the ratio $y_*$ of the two, which by~\eqref{eq:equilibrium} fixes the equilibrium.

\begin{corollary}[Explicit form of the corridor inequality]\label{cor:explicit-corridor}
Under~\eqref{eq:normalization} the constants~\eqref{eq:corridor-constants} are
\begin{equation}\label{eq:constants-explicit}
m=\frac{(1-\varepsilon)-\varepsilon y_*}{1+y_*},
\qquad
\gamma=2|\alpha_*|\bigl[(1-\varepsilon)+\varepsilon y_*^2\bigr] .
\end{equation}
In particular $m>0$, that is $x_*\in(\varepsilon,1-\varepsilon)$, if and only if the two channels are not too unbalanced,
\begin{equation}\label{eq:balance}
y_*<\frac{1-\varepsilon}{\varepsilon} ,
\end{equation}
and the corridor inequality~\eqref{eq:corridor} holds whenever the strength of the weaker channel satisfies
\begin{equation}\label{eq:corridor-explicit}
|\alpha_*|\;>\;\frac{(\bar\delta_\alpha+\bar\delta_\beta)(1-\varepsilon)^2(1+y_*)}
{2\bigl[(1-\varepsilon)-\varepsilon y_*\bigr]} .
\end{equation}
If moreover the equilibrium keeps room $\varepsilon$ inside the interval, that is
\begin{equation}\label{eq:balance-strong}
x_*\le 1-2\varepsilon,
\qquad\text{equivalently}\qquad
y_*\le\frac{1-2\varepsilon}{2\varepsilon}
\qquad(\text{so that }\varepsilon\le 1/4),
\end{equation}
then the ratio drops out of~\eqref{eq:corridor-explicit} and it is enough that
\begin{equation}\label{eq:corridor-simple}
|\alpha_*|\;>\;\frac{(\bar\delta_\alpha+\bar\delta_\beta)(1-\varepsilon)^2}{2\varepsilon} .
\end{equation}
\end{corollary}

\begin{proof}
The function $x\mapsto|\alpha_*|x+|\beta_*|(1-x)$ is affine, so its minimum over $[\varepsilon,1-\varepsilon]$ is attained at an endpoint; the two endpoint values differ by $(2\varepsilon-1)\bigl(|\alpha_*|-|\beta_*|\bigr)\ge 0$ because $\varepsilon<1/2$ and $|\alpha_*|\le|\beta_*|$, so the minimum is the value at $x=1-\varepsilon$, namely $|\alpha_*|(1-\varepsilon)+|\beta_*|\varepsilon$. Writing $|\beta_*|=|\alpha_*|y_*^2$ turns it into the $\gamma$ of~\eqref{eq:constants-explicit}.

For $m$, note that $y_*\ge 1$ gives $x_*\ge 1/2$, so $x_*$ is at least as close to $1-\varepsilon$ as to $\varepsilon$, and by~\eqref{eq:equilibrium}
\[
m=1-\varepsilon-x_*=1-\varepsilon-\frac{y_*}{1+y_*}
=\frac{(1-\varepsilon)-\varepsilon y_*}{1+y_*} ,
\]
which is positive exactly when~\eqref{eq:balance} holds.

Next, $(1-\varepsilon)+\varepsilon y_*^2\ge 1$ because $y_*\ge 1$, so $\gamma\ge2|\alpha_*|$ and $\gamma m\ge 2|\alpha_*|m$. By the formula for $m$ just obtained, the right-hand side of~\eqref{eq:corridor-explicit} is $(\bar\delta_\alpha+\bar\delta_\beta)(1-\varepsilon)^2/2m$, so~\eqref{eq:corridor-explicit} states that $2|\alpha_*|m>(\bar\delta_\alpha+\bar\delta_\beta)(1-\varepsilon)^2$, which therefore implies~\eqref{eq:corridor}.

Finally, $m\ge\varepsilon$ is equivalent to $(1-\varepsilon)-\varepsilon y_*\ge\varepsilon(1+y_*)$, that is to $y_*\le(1-2\varepsilon)/2\varepsilon$, and to $x_*\le1-2\varepsilon$; with $x_*\ge1/2$ this forces $\varepsilon\le1/4$. Under~\eqref{eq:balance-strong} the right-hand side of~\eqref{eq:corridor-explicit} equals $(\bar\delta_\alpha+\bar\delta_\beta)(1-\varepsilon)^2/2m$ and is at most $(\bar\delta_\alpha+\bar\delta_\beta)(1-\varepsilon)^2/2\varepsilon$, so~\eqref{eq:corridor-simple} implies~\eqref{eq:corridor-explicit}.
\end{proof}

Corollary~\ref{cor:explicit-corridor} splits the corridor condition into two requirements on the pair of channels: one on their symmetry, one on their strength.

The first, \eqref{eq:balance}, constrains only the ratio $y_*$ of the two channels and says nothing about their size: the imbalance must stay below $(1-\varepsilon)/\varepsilon$, which forces $y_*$ close to $1$ only when $\varepsilon$ is close to $1/2$, while a modest margin tolerates a markedly unequal pair. Symmetry is what creates room around the equilibrium. Under~\eqref{eq:normalization} the stronger channel is $\beta_*$, which pushes the share upwards, so the equilibrium lies in the upper half of the interval and the binding boundary is $B=0$. The larger $y_*$, the closer the equilibrium sits to that boundary and the smaller the room $m$ around it, and the larger the weaker channel $|\alpha_*|$ must be to hold the share inside the margin: by~\eqref{eq:corridor-explicit} the bound on $|\alpha_*|$ grows without limit as $y_*$ approaches $(1-\varepsilon)/\varepsilon$, and past that value no $|\alpha_*|$ suffices.

The second, \eqref{eq:corridor-simple}, constrains only the size of the weaker channel, once the ratio is kept away from that limit by~\eqref{eq:balance-strong}: the weaker channel must exceed the total size $\bar\delta_\alpha+\bar\delta_\beta$ of the perturbations by the factor $(1-\varepsilon)^2/2\varepsilon$. The factor grows as $\varepsilon$ decreases because \eqref{eq:corridor-simple} covers every ratio that~\eqref{eq:balance-strong} admits, and the smaller $\varepsilon$ is, the wider that range. The binding case is its upper end $y_*=(1-2\varepsilon)/2\varepsilon$, where the room around the equilibrium is exactly $m=\varepsilon$ and the bound on $|\alpha_*|$ in~\eqref{eq:corridor-explicit} rises to $(\bar\delta_\alpha+\bar\delta_\beta)(1-\varepsilon)^2/2\varepsilon$, which is the right-hand side of~\eqref{eq:corridor-simple}.

The two requirements weigh differently. For the share to keep its margin, the second is what matters: by~\eqref{eq:corridor-explicit} every ratio that~\eqref{eq:balance} admits can be sustained, provided the weaker channel is strong enough to outweigh the perturbations. The first tightens only as the margin grows, and demands near-symmetry only as $\varepsilon$ approaches $1/2$. What breaks both at once is the degeneration of a single channel: as $\alpha_*\to 0$ the ratio $y_*$ grows without bound and the weaker channel vanishes, which is the quantitative form of the degenerate cases of Theorem~\ref{thm:classification}. At the equilibrium itself the contraction rate is $|f_*'(x_*)|=2\sqrt{\alpha_*\beta_*}$, the same rate the linearization of Theorem~\ref{thm:classification} gives, so the geometric mean of the two channels governs the local dynamics, while their minimum governs the whole interval.

In the plane of the nominal parameters the admissible set therefore contains the subquadrant
\[
|\alpha_*|>c,\qquad |\beta_*|>c,
\qquad
c:=\frac{(\bar\delta_\alpha+\bar\delta_\beta)(1-\varepsilon)^2}{2\varepsilon} ,
\]
cut only by the ratio bound~\eqref{eq:balance-strong}, which a modest margin leaves loose. This is the stable quadrant of Theorem~\ref{thm:classification} translated off the origin, not a corridor: robustness against a changing environment costs only that neither channel degenerate. The narrowness has another source, and Section~\ref{sec5} identifies it as the demand for growth.

\section{Exponential Growth and the Narrowness of the Corridor}\label{sec5}

The results above concern the composition of the total; this section returns to the original system~\eqref{eq:OL} and compares the growth of the absolute sizes in the three stable regimes. Write~\eqref{eq:OL} as the linear system
\begin{equation}\label{eq:linear-OL}
\dot z=Sz,
\qquad
z(t)=\begin{pmatrix}R(t)\\ B(t)\end{pmatrix},
\qquad
S=\begin{pmatrix}0&-\beta\\ -\alpha&0\end{pmatrix} .
\end{equation}

\subsection{The cooperation premium}\label{subsec5.1}

\begin{theorem}[Exponential cooperation premium]\label{thm:premium}
Let $\alpha\beta\ge 0$ in~\eqref{eq:linear-OL}.
\begin{enumerate}
\item The solution is $z(t)=e^{tS}z(0)$, and $S^2=\alpha\beta I$, hence
\[
e^{tS}=
\begin{cases}
\displaystyle\cosh\left(\sqrt{\alpha\beta}\,t\right)I
+\frac{\sinh\left(\sqrt{\alpha\beta}\,t\right)}{\sqrt{\alpha\beta}}\,S,&\alpha\beta>0,\\[3mm]
I+tS,&\alpha\beta=0 .
\end{cases}
\]
\item Let $\alpha\beta>0$. The eigenvalues of $S$ are $\pm\sqrt{\alpha\beta}$, and for every $z(0)$ outside the eigendirection of $-\sqrt{\alpha\beta}$
\[
\lim_{t\to\infty}\frac1t\log\|z(t)\|=\sqrt{\alpha\beta} .
\]
For $\alpha<0,\ \beta<0$ that eigendirection has components of opposite signs, so the limit equals $\sqrt{\alpha\beta}$ for every $z(0)\neq0$ with $R(0),B(0)\ge 0$.
\item If $\alpha=0$ or $\beta=0$, then $\|z(t)\|$ grows at most linearly.
\item Let $\alpha\beta>0$ with $z(0)$ as in item~2, and let $\tilde z$ be a solution of~\eqref{eq:OL} for parameters $\tilde\alpha,\tilde\beta$ with $\tilde\alpha\tilde\beta=0$ and $\tilde z(0)\neq 0$. Then
\[
\lim_{t\to\infty}\frac1t\log\frac{\|z(t)\|}{\|\tilde z(t)\|}=\sqrt{\alpha\beta}>0 .
\]
\end{enumerate}
\end{theorem}

\begin{proof}
A direct computation gives $S^2=\alpha\beta I$, hence $S^{2k}=(\alpha\beta)^kI$ and $S^{2k+1}=(\alpha\beta)^kS$. Splitting the exponential series into even and odd parts,
\[
e^{tS}=\left(\sum_{k=0}^\infty\frac{(\alpha\beta)^kt^{2k}}{(2k)!}\right)I
+\left(\sum_{k=0}^\infty\frac{(\alpha\beta)^kt^{2k+1}}{(2k+1)!}\right)S .
\]
For $\alpha\beta=0$ only the terms with $k=0$ survive, so $e^{tS}=I+tS$, which proves item~1 in that case and item~3. For $\alpha\beta>0$ put $\lambda:=\sqrt{\alpha\beta}$; the two series are $\cosh(\lambda t)$ and $\lambda^{-1}\sinh(\lambda t)$, which proves item~1 in the remaining case.

For item~2 let $\alpha\beta>0$ and $\lambda_\pm=\pm\sqrt{\alpha\beta}$ be the eigenvalues of $S$, with eigenvectors $v_\pm$. Writing $z(0)=c_+v_++c_-v_-$ gives $z(t)=c_+e^{\lambda_+t}v_++c_-e^{\lambda_-t}v_-$, so $\|z(t)\|$ grows as $e^{\sqrt{\alpha\beta}t}$ whenever $c_+\neq0$, which is the stated limit. For $\alpha<0,\ \beta<0$ the eigenvectors are $v_\pm=(\sqrt{|\beta|},\pm\sqrt{|\alpha|})^{\mathsf T}$ and
\[
c_+=\frac12\left(\frac{R(0)}{\sqrt{|\beta|}}+\frac{B(0)}{\sqrt{|\alpha|}}\right)>0
\]
for every $z(0)\neq 0$ with $R(0),B(0)\ge 0$. The exceptional direction $v_-$ has components of opposite signs and therefore does not meet the domain.

For item~4 write $\tilde S$ for the matrix of~\eqref{eq:linear-OL} at the parameters $\tilde\alpha,\tilde\beta$. Then $\tilde S^2=\tilde\alpha\tilde\beta I=0$, so by item~1 the solution is $\tilde z(t)=\tilde z(0)+t\tilde S\tilde z(0)$. If $\tilde S\tilde z(0)=0$, then $\tilde z(t)=\tilde z(0)\ne 0$ is constant; otherwise $\|\tilde z(t)\|$ grows linearly and does not vanish, because $\tilde S$ has no nonzero eigenvalue. In both cases $\frac1t\log\|\tilde z(t)\|\to0$, and subtracting this from the limit of item~2 gives the claim.
\end{proof}

Item~4 of Theorem~\ref{thm:premium} is the cooperation premium: mutual cooperation is asymptotically stronger than either degenerate regime, and the gap grows exponentially in time. For $\alpha,\beta<0$ the direction excluded in item~2 is $R/B=-y_*$, and the premium is therefore collected from every admissible initial state.

\subsection{The cost of cooperation}\label{subsec5.2}

The margin of Section~\ref{sec4} turns this premium into a guaranteed rate. Take the nominal parameters $\alpha_*,\beta_*<0$ of~\eqref{eq:perturbed-params}. By~\eqref{eq:Ndot} the total of the original system obeys $\dot N=-\mu N$ with
\[
-\mu=|\alpha_*|x+|\beta_*|(1-x) ,
\]
a convex combination of the two channels, so while the share keeps the margin $\varepsilon$ the rate is at least $\gamma/2$ by~\eqref{eq:corridor-constants}. It is in fact largest at $x=0$, where the stronger channel acts alone, but that point cannot be held, since $\dot x|_{x=0}=|\beta_*|>0$ carries the share back. What the margin secures is the sustained rate, the equilibrium value $\sqrt{|\alpha_*||\beta_*|}$ of item~2 in Theorem~\ref{thm:premium}.

That rate is what the pair produces. What it costs is the sum $|\alpha_*|+|\beta_*|$: both channels are supplied by the same population, since the state is staffed from the society it governs, and two demands on one stock deplete it additively. Any aggregate that stays positive when one channel vanishes would serve as well, and the geometric mean is the one that does not: a society whose citizens serve a state that gives nothing back consumes cooperation and produces no growth, which is the degenerate case of item~3 in Theorem~\ref{thm:premium}.

Growth per unit of cost is therefore
\begin{equation}\label{eq:conversion}
q:=\frac{\sqrt{|\alpha_*||\beta_*|}}{|\alpha_*|+|\beta_*|}=\frac{y_*}{1+y_*^{2}} ,
\end{equation}
a function of the imbalance alone, largest at $y_*=1$ where it equals $1/2$. By~\eqref{eq:balance} the widest margin the equilibrium admits, $\varepsilon_{\max}=1-x_*=1/(1+y_*)$, falls with $y_*$ in the same way, so as one channel degenerates the margin, the efficiency $q$ and the premium vanish together. Being a ratio, $q$ is unchanged when both channels are scaled together: it is a structural property of the pair, and demanding growth of a given cost is exactly a bound on it. That bound, together with the lower bound the perturbations put on the size of the channels, cuts a corridor out of the quadrant, whose boundaries the next corollary writes out.

\subsection{The narrow corridor}\label{subsec5.3}

\begin{corollary}[The wedge of sustained growth]\label{cor:wedge}
Let the parameters be as in~\eqref{eq:perturbed-params} with $\alpha_*,\beta_*<0$, put $\bar\delta:=\bar\delta_\alpha+\bar\delta_\beta$, and let the imbalance be bounded by some $\bar y\ge 1$, that is, let the two channels lie in the wedge
\begin{equation}\label{eq:wedge}
\frac{1}{\bar y^{2}}\;\le\;\frac{|\beta_*|}{|\alpha_*|}\;\le\;\bar y^{2} ,
\qquad\text{equivalently}\qquad
\frac{1}{\bar y}\le y_*\le\bar y ,
\end{equation}
which opens to the whole quadrant as $\bar y\to\infty$ and closes onto the diagonal at $\bar y=1$. Put
\begin{equation}\label{eq:wedge-constants}
\bar\varepsilon:=\frac{1}{2(1+\bar y)}\ \le\frac14,
\qquad
\underline c:=\bar\delta\,\frac{(1+2\bar y)^2}{4(1+\bar y)} .
\end{equation}
Then
\begin{enumerate}
\item on the wedge~\eqref{eq:wedge} the balance condition~\eqref{eq:balance-strong} holds for the margin $\bar\varepsilon$, and at the edges of the wedge it fails for every larger margin;
\item if moreover the two channels exceed the truncation height,
\begin{equation}\label{eq:truncation}
\min\bigl(|\alpha_*|,|\beta_*|\bigr)>\underline c ,
\end{equation}
then the corridor inequality~\eqref{eq:corridor} holds with $\varepsilon=\bar\varepsilon$, so by Theorem~\ref{thm:robust} every solution with $|x(0)-x_*|\le m$, where $m$ is the constant~\eqref{eq:corridor-constants} at the margin $\bar\varepsilon$, keeps that margin for all $t\ge 0$;
\item in that case $\alpha(t)<0$ and $\beta(t)<0$ under every perturbation, and for every solution as in item~2 the total of the original system~\eqref{eq:OL} obeys $N(t)\ge N(0)e^{\nu t}$ with $\nu:=\gamma/2-\bar\delta>0$, where $\gamma$ is the constant~\eqref{eq:corridor-constants} at the margin $\bar\varepsilon$.
\end{enumerate}
\end{corollary}

\begin{proof}
1. Condition~\eqref{eq:balance-strong} for a margin $\varepsilon$ reads $y_*\le(1-2\varepsilon)/2\varepsilon=1/2\varepsilon-1$, whose right-hand side at $\varepsilon=\bar\varepsilon$ is $(1+\bar y)-1=\bar y$. It therefore holds throughout~\eqref{eq:wedge}, with equality at $y_*=\bar y$, and fails there for every $\varepsilon>\bar\varepsilon$.

2. Take the normalization~\eqref{eq:normalization}, so that $\min(|\alpha_*|,|\beta_*|)=|\alpha_*|$. Since $1-\bar\varepsilon=(1+2\bar y)/2(1+\bar y)$,
\[
\frac{\bar\delta(1-\bar\varepsilon)^2}{2\bar\varepsilon}
=\bar\delta\,\frac{(1+2\bar y)^2}{4(1+\bar y)^2}\,(1+\bar y)=\underline c ,
\]
so~\eqref{eq:truncation} is exactly~\eqref{eq:corridor-simple} at the margin $\bar\varepsilon$. By item~1 the hypothesis~\eqref{eq:balance-strong} of Corollary~\ref{cor:explicit-corridor} holds, so~\eqref{eq:corridor-simple} gives~\eqref{eq:corridor-explicit} and hence~\eqref{eq:corridor}, and Theorem~\ref{thm:robust} applies with $\varepsilon=\bar\varepsilon$.

3. Because $(1+2\bar y)^2-4(1+\bar y)=4\bar y^2-3>0$, the constant satisfies $\underline c>\bar\delta$, so both channels exceed $\bar\delta\ge\max(\bar\delta_\alpha,\bar\delta_\beta)$ and stay negative. By item~2 the share stays in $[\bar\varepsilon,1-\bar\varepsilon]$, where $|\alpha(t)|\ge|\alpha_*|-\bar\delta_\alpha$ and $|\beta(t)|\ge|\beta_*|-\bar\delta_\beta$ give, by~\eqref{eq:Ndot} and~\eqref{eq:corridor-constants},
\[
-\mu(t)=|\alpha(t)|x+|\beta(t)|(1-x)
\ge|\alpha_*|x+|\beta_*|(1-x)-\bar\delta
\ge\frac{\gamma}{2}-\bar\delta .
\]
Here $\gamma/2$ is a convex combination of $|\alpha_*|$ and $|\beta_*|$, hence at least $\min(|\alpha_*|,|\beta_*|)>\underline c>\bar\delta$, so $\nu>0$. Integrating $\dot N=-\mu N$ gives $N(t)\ge N(0)e^{\nu t}$.
\end{proof}

Narrowness is therefore the price of growth, not of stability. The region~\eqref{eq:wedge}--\eqref{eq:truncation} is a corridor in the literal sense (Figure~\ref{fig:qual}), cut out of a quadrant that Theorem~\ref{thm:classification} makes stable throughout, and two independent numbers draw it: the bound $\bar y$ on the imbalance opens the wedge, and the perturbation bound $\bar\delta$ sets the height $\underline c$ at which it is truncated. They move in opposite directions, since lowering $\bar y$ thins the wedge and at the same time lowers its truncation, to $9\bar\delta/8$ on the diagonal, where near-equal channels hold the equilibrium in the middle of the interval and the perturbations have least purchase on it.

\begin{figure}[ht]
\centering
\begin{tikzpicture}[baseline=0pt, scale=4.4545, >=stealth, every node/.style={font=\footnotesize}]

  \fill[blue!55, fill opacity=0.30]
      (0.30,0.51) -- (0.30,0.30) -- (0.51,0.30) -- (1.10,0.647)
                  -- (1.10,1.10) -- (0.647,1.10) -- cycle;
  \draw[blue!60!black, thick] (0,0) -- (0.30,0.30);
  \draw[blue!60!black, thick] (0.30,0.30) -- (0.30,0.51) -- (0.647,1.10);
  \draw[blue!60!black, thick] (0.30,0.30) -- (0.51,0.30) -- (1.10,0.647);

  \draw[->] (0,0) -- (1.20,0);
  \draw[->] (0,0) -- (0,1.20);
  \node at (0.60,-0.067) {Power of the Society};
  \node[rotate=90] at (-0.067,0.60) {Power of the State};

  \node[align=center] at (0.18,1.00) {\textbf{Despotic}\\\textbf{Leviathan}};
  \node[align=center] at (0.95,0.30) {\textbf{Absent}\\\textbf{Leviathan}};
  \node[align=center] at (0.30,0.09) {\textbf{Paper}\\\textbf{Leviathan}};

  \node[rotate=45, align=center] at (0.88,0.88) {\textbf{Shackled}\\\textbf{Leviathan}};

  \fill[blue!50!black] (0.23,0.42) circle (0.011);
  \draw[->,blue!50!black,thick] (0.23,0.42) to[out=110,in=-80] (0.14,0.76);

  \fill[blue!50!black] (0.48,0.20) circle (0.011);
  \draw[->,blue!50!black,thick] (0.48,0.20) to[out=-20,in=180] (0.78,0.13);

  \draw[->,blue!50!black,thick] (0.58,0.60) -- (0.74,0.76);

\end{tikzpicture}%
\hfill
\begin{tikzpicture}[baseline=-2.45cm, scale=2.45, >=stealth, every node/.style={font=\footnotesize}]

  \fill[green!20] (0,0) rectangle (1,1);
  \node at (0.52,1.10) {\textbf{democracy}};
  \draw[->,thin] (0.52,1.04) -- (0.45,0.95);

  \fill[red!6] (-1,0) rectangle (0,1);
  \fill[red!6] (-1,-1) rectangle (0,0);
  \fill[red!6] (0,-1) rectangle (1,0);

  \node[rotate=45] at (-0.5,0.533) {\textbf{confrontation}};
  \node[rotate=45] at (-0.5,-0.533) {\textbf{confrontation}};
  \node[rotate=45] at (0.5,-0.533) {\textbf{confrontation}};

  \fill[orange!40, draw=orange!80!black] (0,0) rectangle (0.133,1);
  \node[rotate=90] at (0.067,0.5) {\textbf{autocracy}};

  \fill[orange!40, draw=orange!80!black] (0,0) rectangle (1,0.133);
  \node at (0.5,0.067) {\textbf{anarchy}};

  \filldraw[fill=blue!55, fill opacity=0.30, draw=blue!60!black, dashed, thick]
      (0.2042,0.3191) -- (0.2042,0.2042) -- (0.3191,0.2042)
      -- (1,0.64) -- (1,1) -- (0.64,1) -- cycle;

  \draw[thick] (-1,-1) rectangle (1,1);
  \draw[->] (-1,0) -- (1.19,0) node[below] {$-\alpha_*$};
  \draw[->] (0,-1) -- (0,1.22) node[left] {$-\beta_*$};

  \node[rotate=45] at (0.62,0.62) {\textbf{sustained growth}};

  \node at (0.5,-0.067) {inclusiveness};
  \node at (-0.5,-0.067) {extractiveness};
  \node[rotate=90] at (-0.067,0.5) {patriotism};
  \node[rotate=90] at (-0.067,-0.5) {disloyalty};
\end{tikzpicture}
\caption{(left) The narrow corridor of~\cite{AcemogluRobinson2019}, drawn schematically as a wedge with a tail, (right) regimes in the $(-\alpha_*,-\beta_*)$ axes. The shaded region is the wedge~\eqref{eq:wedge} truncated by~\eqref{eq:truncation}, drawn for the imbalance bound $\bar y=1.25$ and the perturbation bound $\bar\delta=0.15$ on axes running over $-\alpha_*,-\beta_*\in[-1,1]$.}\label{fig:qual}
\end{figure}

\section{Conclusion}\label{sec6}

The paper extends the Osipov--Lanchester model to all values of $\alpha$ and $\beta$ by fixing the total population $N$. The extension describes both confrontation and cooperation, and the reduction to the Riccati equation gives a complete classification of the stable invariant sets. The model is minimal and qualitative; calibration to data is a separate task.

The regimes admit an interpretation in the terms of~\cite{AcemogluRobinson2019,acemoglu2012whynationsfail}. Let $R$ be the strength of the state and $B$ the strength of society. For $\alpha\le 0$ and $\beta\le 0$ the magnitudes $|\alpha|$ and $|\beta|$ measure the inclusiveness of the state and the patriotism of society. The three stable regimes of Theorem~\ref{thm:classification} then read as follows (Figure~\ref{fig:qual}):
\begin{enumerate}
\item $\alpha<0,\ \beta<0$: stable balance, or ``democracy'';
\item $\alpha=0,\ \beta<0$: the stable boundary $B=0$, or ``autocracy'';
\item $\alpha<0,\ \beta=0$: the stable boundary $R=0$, or ``anarchy''.
\end{enumerate}
Here ``autocracy'' is a broad label for state-dominant regimes, which includes the despotic Leviathan of~\cite{AcemogluRobinson2019}, and ``democracy'' means any form of stable cooperation between the state and society. The extreme regimes become attractors when one of the two channels degenerates, and confrontation leads outside the domain where the model applies. In particular there is no ``end of history'' in the sense of~\cite{Fukuyama1992EndOfHistoryLastMan}.

Stability by itself is not demanding: by Corollary~\ref{cor:explicit-corridor} a modest margin tolerates channels of very different strength, so most of the quadrant $\alpha\le 0,\ \beta\le 0$ will serve. Two things restrict the parameters further, their change in time and the demand for growth, and the two act differently. The first acts near the origin: by Corollary~\ref{cor:buffer} the share keeps a fixed distance from the boundaries when the parameters obey buffer dynamics that push them away from zero, and by Theorem~\ref{thm:robust} the admissible perturbations are the smaller the weaker the two channels. What this forbids is the degeneration of a channel, nothing more. In institutional terms, mechanisms of power rotation and mutual constraints are what holds the parameters inside the robust subquadrant.

Finally, by Theorem~\ref{thm:premium} the absolute sizes grow exponentially for $\alpha<0,\ \beta<0$ and at most linearly in the two degenerate regimes. The rate is the geometric mean of the two channels while their cost is their sum, so at a given cost the rate is largest when they are equal and falls away as the imbalance grows; growth demanded of a given cost therefore bounds that imbalance, which is the wedge of Corollary~\ref{cor:wedge} and the narrow corridor of the model. Growth demands the symmetry that stability alone does not, and it cannot be bought with one-sided cooperation, however strong, which is the reading of~\cite{AcemogluRobinson2019,acemoglu2012whynationsfail}. The premium lasts only as long as the resources do, which historically has meant technological progress; when it weakens, so does the incentive to stay in the narrow corridor.

\backmatter

\bibliography{main}


\begin{thebibliography}{18}
\ifx \bisbn   \undefined \def \bisbn  #1{ISBN #1}\fi
\ifx \binits  \undefined \def \binits#1{#1}\fi
\ifx \bauthor  \undefined \def \bauthor#1{#1}\fi
\ifx \batitle  \undefined \def \batitle#1{#1}\fi
\ifx \bjtitle  \undefined \def \bjtitle#1{#1}\fi
\ifx \bvolume  \undefined \def \bvolume#1{\textbf{#1}}\fi
\ifx \byear  \undefined \def \byear#1{#1}\fi
\ifx \bissue  \undefined \def \bissue#1{#1}\fi
\ifx \bfpage  \undefined \def \bfpage#1{#1}\fi
\ifx \blpage  \undefined \def \blpage #1{#1}\fi
\ifx \burl  \undefined \def \burl#1{\textsf{#1}}\fi
\ifx \doiurl  \undefined \def \doiurl#1{\url{https://doi.org/#1}}\fi
\ifx \betal  \undefined \def \betal{\textit{et al.}}\fi
\ifx \binstitute  \undefined \def \binstitute#1{#1}\fi
\ifx \binstitutionaled  \undefined \def \binstitutionaled#1{#1}\fi
\ifx \bctitle  \undefined \def \bctitle#1{#1}\fi
\ifx \beditor  \undefined \def \beditor#1{#1}\fi
\ifx \bpublisher  \undefined \def \bpublisher#1{#1}\fi
\ifx \bbtitle  \undefined \def \bbtitle#1{#1}\fi
\ifx \bedition  \undefined \def \bedition#1{#1}\fi
\ifx \bseriesno  \undefined \def \bseriesno#1{#1}\fi
\ifx \blocation  \undefined \def \blocation#1{#1}\fi
\ifx \bsertitle  \undefined \def \bsertitle#1{#1}\fi
\ifx \bsnm \undefined \def \bsnm#1{#1}\fi
\ifx \bsuffix \undefined \def \bsuffix#1{#1}\fi
\ifx \bparticle \undefined \def \bparticle#1{#1}\fi
\ifx \barticle \undefined \def \barticle#1{#1}\fi
\bibcommenthead
\ifx \bconfdate \undefined \def \bconfdate #1{#1}\fi
\ifx \botherref \undefined \def \botherref #1{#1}\fi
\ifx \url \undefined \def \url#1{\textsf{#1}}\fi
\ifx \bchapter \undefined \def \bchapter#1{#1}\fi
\ifx \bbook \undefined \def \bbook#1{#1}\fi
\ifx \bcomment \undefined \def \bcomment#1{#1}\fi
\ifx \oauthor \undefined \def \oauthor#1{#1}\fi
\ifx \citeauthoryear \undefined \def \citeauthoryear#1{#1}\fi
\ifx \endbibitem  \undefined \def \endbibitem {}\fi
\ifx \bconflocation  \undefined \def \bconflocation#1{#1}\fi
\ifx \arxivurl  \undefined \def \arxivurl#1{\textsf{#1}}\fi
\csname PreBibitemsHook\endcsname

\bibitem[\protect\citeauthoryear{Acemoglu and
  Robinson}{2019}]{AcemogluRobinson2019}
\begin{bbook}
\bauthor{\bsnm{Acemoglu}, \binits{D.}},
\bauthor{\bsnm{Robinson}, \binits{J.A.}}:
\bbtitle{The Narrow Corridor: States, Societies, and the Fate of Liberty}.
\bpublisher{Penguin Press},
\blocation{New York}
(\byear{2019})
\end{bbook}
\endbibitem

\bibitem[\protect\citeauthoryear{Lanchester}{1916}]{Lanchester1916}
\begin{bbook}
\bauthor{\bsnm{Lanchester}, \binits{F.W.}}:
\bbtitle{Aircraft in Warfare: The Dawn of the Fourth Arm}.
\bpublisher{Constable},
\blocation{London}
(\byear{1916})
\end{bbook}
\endbibitem

\bibitem[\protect\citeauthoryear{Helmbold}{1993}]{Helmbold1993}
\begin{barticle}
\bauthor{\bsnm{Helmbold}, \binits{R.L.}}:
\batitle{Osipov: The `{R}ussian {L}anchester'}.
\bjtitle{European Journal of Operational Research}
\bvolume{65}(\bissue{2}),
\bfpage{278}--\blpage{288}
(\byear{1993})
\doiurl{10.1016/0377-2217(93)90341-J}
\end{barticle}
\endbibitem

\bibitem[\protect\citeauthoryear{Arnold}{2000}]{Arnold2000HardSoft}
\begin{bbook}
\bauthor{\bsnm{Arnold}, \binits{V.I.}}:
\bbtitle{``Zhestkie'' i ``myagkie'' {matematicheskie} {modeli} [``Hard'' and
  ``soft'' {mathematical} {models}]}.
\bpublisher{MCCME},
\blocation{Moscow}
(\byear{2000}).
\bcomment{In Russian}
\end{bbook}
\endbibitem

\bibitem[\protect\citeauthoryear{Deitchman}{1962}]{Deitchman1962}
\begin{barticle}
\bauthor{\bsnm{Deitchman}, \binits{S.J.}}:
\batitle{A {L}anchester model of guerrilla warfare}.
\bjtitle{Operations Research}
\bvolume{10}(\bissue{6}),
\bfpage{818}--\blpage{827}
(\byear{1962})
\doiurl{10.1287/opre.10.6.818}
\end{barticle}
\endbibitem

\bibitem[\protect\citeauthoryear{Kress}{2020}]{Kress2020Irregular}
\begin{barticle}
\bauthor{\bsnm{Kress}, \binits{M.}}:
\batitle{{L}anchester models for irregular warfare}.
\bjtitle{Mathematics}
\bvolume{8}(\bissue{5}),
\bfpage{737}
(\byear{2020})
\doiurl{10.3390/math8050737}
\end{barticle}
\endbibitem

\bibitem[\protect\citeauthoryear{Kress et~al.}{2018}]{Kress2018Threeway}
\begin{barticle}
\bauthor{\bsnm{Kress}, \binits{M.}},
\bauthor{\bsnm{Caulkins}, \binits{J.P.}},
\bauthor{\bsnm{Feichtinger}, \binits{G.}},
\bauthor{\bsnm{Grass}, \binits{D.}},
\bauthor{\bsnm{Seidl}, \binits{A.}}:
\batitle{{L}anchester model for three-way combat}.
\bjtitle{European Journal of Operational Research}
\bvolume{264}(\bissue{1}),
\bfpage{46}--\blpage{54}
(\byear{2018})
\doiurl{10.1016/j.ejor.2017.07.026}
\end{barticle}
\endbibitem

\bibitem[\protect\citeauthoryear{Volpert}{2024}]{Volpert2024}
\begin{barticle}
\bauthor{\bsnm{Volpert}, \binits{V.A.}}:
\batitle{Mathematical model of repressive response to collective action and
  protest waves}.
\bjtitle{Journal of Theoretical Biology}
\bvolume{595},
\bfpage{111970}
(\byear{2024})
\doiurl{10.1016/j.jtbi.2024.111970}
\end{barticle}
\endbibitem

\bibitem[\protect\citeauthoryear{Turchin}{2003}]{Turchin2003}
\begin{bbook}
\bauthor{\bsnm{Turchin}, \binits{P.}}:
\bbtitle{Historical Dynamics: Why States Rise and Fall}.
\bpublisher{Princeton University Press},
\blocation{Princeton, NJ}
(\byear{2003})
\end{bbook}
\endbibitem

\bibitem[\protect\citeauthoryear{Dockner
  et~al.}{2000}]{DocknerJorgensenLongSorger2000}
\begin{bbook}
\bauthor{\bsnm{Dockner}, \binits{E.J.}},
\bauthor{\bsnm{J{\o}rgensen}, \binits{S.}},
\bauthor{\bsnm{Long}, \binits{N.V.}},
\bauthor{\bsnm{Sorger}, \binits{G.}}:
\bbtitle{Differential Games in Economics and Management Science}.
\bpublisher{Cambridge University Press},
\blocation{Cambridge}
(\byear{2000})
\end{bbook}
\endbibitem

\bibitem[\protect\citeauthoryear{Yeung and
  Petrosyan}{2012}]{YeungPetrosyan2012}
\begin{bbook}
\bauthor{\bsnm{Yeung}, \binits{D.W.K.}},
\bauthor{\bsnm{Petrosyan}, \binits{L.A.}}:
\bbtitle{Subgame Consistent Economic Optimization: An Advanced Cooperative
  Dynamic Game Analysis}.
\bpublisher{Birkh{\"a}user},
\blocation{Boston}
(\byear{2012}).
\doiurl{10.1007/978-0-8176-8262-0}
\end{bbook}
\endbibitem

\bibitem[\protect\citeauthoryear{Bomze}{1995}]{Bomze1995NewIssues}
\begin{barticle}
\bauthor{\bsnm{Bomze}, \binits{I.M.}}:
\batitle{{L}otka--{V}olterra equation and replicator dynamics: new issues in
  classification}.
\bjtitle{Biological Cybernetics}
\bvolume{72}(\bissue{5}),
\bfpage{447}--\blpage{453}
(\byear{1995})
\doiurl{10.1007/BF00201420}
\end{barticle}
\endbibitem

\bibitem[\protect\citeauthoryear{Hofbauer and
  Sigmund}{1998}]{HofbauerSigmund1998}
\begin{bbook}
\bauthor{\bsnm{Hofbauer}, \binits{J.}},
\bauthor{\bsnm{Sigmund}, \binits{K.}}:
\bbtitle{Evolutionary Games and Population Dynamics}.
\bpublisher{Cambridge University Press},
\blocation{Cambridge}
(\byear{1998}).
\doiurl{10.1017/CBO9781139173179}
\end{bbook}
\endbibitem

\bibitem[\protect\citeauthoryear{Sandholm}{2010}]{Sandholm2010}
\begin{bbook}
\bauthor{\bsnm{Sandholm}, \binits{W.H.}}:
\bbtitle{Population Games and Evolutionary Dynamics}.
\bpublisher{MIT Press},
\blocation{Cambridge, MA}
(\byear{2010})
\end{bbook}
\endbibitem

\bibitem[\protect\citeauthoryear{Arnold}{1992}]{Arnold1992OrdinaryDifferentialEquations}
\begin{bbook}
\bauthor{\bsnm{Arnold}, \binits{V.I.}}:
\bbtitle{Ordinary Differential Equations},
\bedition{3}rd edn.
\bsertitle{Universitext}.
\bpublisher{Springer},
\blocation{Berlin}
(\byear{1992})
\end{bbook}
\endbibitem

\bibitem[\protect\citeauthoryear{Khalil}{2002}]{khalil2002nonlinear}
\begin{bbook}
\bauthor{\bsnm{Khalil}, \binits{H.K.}}:
\bbtitle{Nonlinear Systems},
\bedition{3}rd edn.
\bpublisher{Prentice Hall},
\blocation{Upper Saddle River, NJ}
(\byear{2002})
\end{bbook}
\endbibitem

\bibitem[\protect\citeauthoryear{Acemoglu and
  Robinson}{2012}]{acemoglu2012whynationsfail}
\begin{bbook}
\bauthor{\bsnm{Acemoglu}, \binits{D.}},
\bauthor{\bsnm{Robinson}, \binits{J.A.}}:
\bbtitle{Why Nations Fail: The Origins of Power, Prosperity, and Poverty}.
\bpublisher{Crown Business},
\blocation{New York}
(\byear{2012})
\end{bbook}
\endbibitem

\bibitem[\protect\citeauthoryear{Fukuyama}{1992}]{Fukuyama1992EndOfHistoryLastMan}
\begin{bbook}
\bauthor{\bsnm{Fukuyama}, \binits{F.}}:
\bbtitle{The End of History and the Last Man}.
\bpublisher{The Free Press},
\blocation{New York}
(\byear{1992})
\end{bbook}
\endbibitem

\end{thebibliography}

\end{document}